\documentclass[12pt, reqno]{amsart}
\usepackage{amsmath, amsthm, amsfonts, amssymb, graphicx, color, mathrsfs}
\usepackage{enumerate}
\usepackage[bookmarksnumbered, colorlinks, plainpages]{hyperref}

\newtheorem{theorem}{\bf Theorem}[section]
\newtheorem{lemma}[theorem]{\bf Lemma}
\newtheorem{proposition}[theorem]{\bf Proposition}
\newtheorem{corollary}[theorem]{\bf Corollary}
\theoremstyle{definition}
\newtheorem{definition}[theorem]{\bf Definition}
\theoremstyle{remark}
\newtheorem{rem}[theorem]{\bf Remark}

\numberwithin{equation}{section}
\newcommand{\mA}{{\mathcal A}}

\newcommand{\sR}{{ \mathscr R }}

\newcommand{\supp}{\text{supp}}

\newcommand{\wsn}{\text{weak$^*$-$\lim_n $}}

\newcommand{\wsa}{\text{weak$^*$-$\lim_\alpha $}}

\newcommand{\wa}{\text{weak-$\lim_\alpha $}}

\newcommand{\1}{{\bf 1}}

\begin{document}
	
	\title[power boundedness property of $B(H)  $ and $A(H)$]{Power boundedness in Fourier and Fourier-Stieltjes
algebra of an ultraspherical hypergroup }

\author[R.  Esmailvandi]{Reza Esmailvandi}
\address{Reza Esmailvandi\newline\indent Department of Mathematical Sciences,\newline\indent Isfahan University of Technology,\newline\indent Isfahan 84156-83111, Iran.}
\email{\textcolor[rgb]{0.00,0.00,0.84}{r.esmailvandi@math.iut.ac.ir} 
}

\author[M. Nemati ]{Mehdi Nemati}
\address{Mehdi Nemati\newline\indent Department of Mathematical Sciences,\newline\indent Isfahan University of Technology,\newline\indent Isfahan 84156-83111, Iran.\newline\indent
and\newline\indent
School of Mathematics,\newline\indent Institute for Research in Fundamental Sciences (IPM),\newline\indent
P.O. Box: 19395--5746, Tehran, Iran}
\email{\textcolor[rgb]{0.00,0.00,0.84}{m.nemati@iut.ac.ir}}

\author[N. S. Kumar]{N. Shravan kumar}
\address{N. Shravan kumar\newline\indent Department of mathematics,\newline\indent Indian Institute of Technology Delhi,\newline\indent Delhi - 110016,\newline\indent India.}
\email{\textcolor[rgb]{0.00,0.00,0.84} {shravankumar.nageswaran@gmail.com}}

	
	
	\begin{abstract}
		Let $H$ be an ultraspherical hypergroup associated with a locally compact group $G$ and a spherical projector $\pi$ and let $A(H)$ and $B(H)$ denote the Fourier and Fourier-Stieltjes algebras, respectively, associated with $H.$ In this note, we study power boundedness and  Ces\`aro boundedness in $B(H)$. We also characterize the power bounded property for both $A(H)$ and $B(H).$
	\end{abstract}
	\maketitle
	
	\section{Introduction}
	Let $G$ be a locally compact abelian group. A classical result of Schreiber \cite{sch70} states that the algebras $L^1(G)$ and $M(G)$ are power bounded if and only if $G$ is compact and finite, respectively. A natural generalizations of the group algebra and the measure algebra are the Fourier and Fourier-Stieltjes algebras associated to general locally compact groups. These algebras are introduced by Eymard \cite{eym64} and since then, they have been a major object of study in abstract harmonic analysis. In 2011, Lau and Kaniuth showed that $A(G)$ and $B(G)$ are power bounded if and only if $G$ is discrete and finite, respectively.
	
	Let $G$ be a locally compact group and let $\pi$ be a spherical projector. Let $H$ be an ultraspherical hypergroup associated to $G$ and $\pi.$ Let $A(H)$ and $B(H)$ denote the Fourier and Fourier-stieltjes algebras on $H,$ introduced by Muruganandam in 2008 \cite{mur08}. The purpose of this paper, is to extend the result of Lau and Kaniuth on groups to the context of ultraspherical hypergroups. More precisely, we show that $A(H)$ and $B(H)$ possess the power bounded property if and only if $H$ is discrete and finite, respectively.
	
	In Section 4, we characterize the power boundedness of $A(H).$ In fact, this is the content of Theorem \ref{CPBP}. Initially, in Section 3, results from \cite{klu10,klu11} are used to deduce relations among power bounded elements of $B(H)$ and certain sets from the coset ring. We also study the sequence $\{u^n\}_{n\in \mathbb{N}}$ in relation to the same sets considered, where $u$ is a power bounded element of $B(H).$ Finally, we study Ces\`aro  boundedness. By making use of the results of Section 3, in Section 4, we characterize the power bounded property for $B(H).$ This is Theorem \ref{CPBPFS}. 
	
	We shall begin with some preliminaries that are needed in the sequel.
	
	\section{preliminaries}
	
	\subsection{Hypergroups} 
	For the definition and basic properties of  hypergroups, we refer the reader mainly to \cite{blh95, jew75}. In \cite{jew75}, Jewett calls a hypergroup as convo. For a hypergroup $ H  $ with the identity element $ e $, the mapping  $ x\mapsto \check{x}$ will always denote the involution of $ H $. For the sake of completeness and
	the convenience of the reader, let us recall the definition  of an ultraspherical hypergroup associated with a locally compact group $G$ (see \cite[Definition 2.1]{mur08} and \cite[Definition 2.2]{dkl14}).
	
	\begin{definition} 
		A linear map $\pi:C_c(G)\rightarrow C_c(G)$ is called a spherical projector if it satisfies the following for all $f,g\in C_c(G).$
		
		\begin{enumerate}
			\item\begin{enumerate}
				\item $\pi^2=\pi$ and $\pi$ is positivity preserving;
				\item $\pi(\pi(f)g)=\pi(f)\pi(g)$;
				\item $\langle \pi(f),g \rangle=\langle f,\pi(g) \rangle$;
				\item $\int_G\pi(f)(x)\ dx=\int_G f(x)\ dx.$
			\end{enumerate}
			\item $\pi(\pi(f)*\pi(g))=\pi(f)*\pi(g)$
			\item Let $\pi^*:M(G)\rightarrow M(G)$ denote the transpose of $\pi$ and let $$\mathcal{O}_x=supp(\pi^*(\delta_x)),\ x\in G.$$ Then for all $x,y\in G,$
			\begin{enumerate}
				\item either $\mathcal{O}_x\cap\mathcal{O}_y=\emptyset$ or $\mathcal{O}_x=\mathcal{O}_y;$
				\item if $x\in\mathcal{O}_y\Rightarrow x^{-1}\in\mathcal{O}_{y^{-1}};$
				\item if $\mathcal{O}_{xy}=\{e\}\Rightarrow\mathcal{O}_y=\mathcal{O}_{x^{-1}};$
				\item the map $x\mapsto\mathcal{O}_x$ from $G$ into $\mathcal{K}(G)$ is continuous with respect
				to the Michael topology on the space $\mathcal{K}(G)$ of all nonempty compact subsets of $ G $ .
			\end{enumerate}
		\end{enumerate}
	\end{definition}
	
	Note that $\pi$ extends to a norm decreasing linear map on various function spaces, including $L^p(G),\ 1\leq p\leq \infty.$ A function $f$ is called $\pi$-radial if $\pi(f)=f$ and similarly, a measure $\mu$ is called $\pi$-radial if $\pi^*(\mu)=\mu.$
	
	Let $ H = \{\dot{x}=\mathcal{O}x: x \in G\}$, equipped with the natural quotient topology under the quotient  map $ p: G\rightarrow H $. We identify $M(H)$ with the space of all  $ \pi^* $-radial measures in $ M(G).$ Define the product on $M(H)$ as $ \delta_{\dot x} \ast  \delta_{\dot y}= \pi^{\ast}(\pi^{\ast}(\delta_{x}) \ast \pi^{\ast}(\delta_{y} )) $ for all $ x , y \in G $. With this structure $ H $ becomes a locally compact hypergroup, called a spherical hypergroup \cite[Theorem 2.12]{mur08}. A spherical hypergroup is further called an ultraspherical hypergroup if the modular function on $G$ is $\pi$-radial.
	
	One of the most common and interesting example of an ultraspherical hypergroup is the double coset hypergroup. Let $G$ be a locally compact group containing a compact subgroup $K.$ Define $\pi:C_c(G)\rightarrow C_c(G)$ as $$\pi(f)(x)=\int_K\int_K f(k^\prime xk)\ dk\ dk^\prime.$$ Then $\pi$ defines a spherical projector and the resulting hypergroup is an ultraspherical hypergroup.  
	
	Specific examples of the double coset hypergroups are orbit hypergroups which are constructed as follows. Let $ G $ be a locally compact group and let $ K $ be a compact subgroup of $ Aut(G) $, the group of all topological automorphisms of $ G $. For each $x \in G$ denote by $ K(x) =\{ \alpha(x) : \alpha \in K\}$ the orbit of $ x $  under $ K $. The orbit hypergroup $ G_{K} $ is the space of all orbits $ K(x), x\in G $, equipped  with the quotient topology and the obvious hypergroup structure. Now,  consider the semidirect product 
	$G \rtimes K$, where the group  product of  $ (x , \alpha)$ with $(y , \beta)$ in $G \rtimes K$ is given by  
	$$(x , \alpha)(y , \beta)= (x \alpha(y) , \alpha \beta).$$ Then the map $ K(x) \rightarrow K(x , \beta) K $ is a toplogical hypergroup isomorphism between $ G_{K} $  and the double coset hypergroup $(G \rtimes K) / /  K  $ (see \cite[Section 8]{jew75}).
	
	Another interesting class of examples for ultraspherical hypergroups in the context of Lie groups is due to Damek and Ricci \cite{dar92}. They called the spherical projectors as average projector.
	
	Unlike the theory of locally compact groups, the existence of a Haar measure on hypergroups
	remains an open problem. But for specific cases of hypergroups including commutative
	hypergroups, discrete hypergroups, and compact hypergroups, a unique (up to a scalar multiple) Haar measure always exists, \cite{blh95}. It is shown in \cite [Proposition 2.14]{mur08}  that every ultraspherical hypergroup admits a Haar measure.
	More precisely, if $H$ is an ultraspherical hypergroup associated to  a locally compact group
	$G$,
	then a left Haar measure on $H$ is given by 
	$$\int_{H} f(\dot{x})d\dot{x}=\int_G f(p(x))dx,
	\qquad  (f\in C_{c}(H)).$$ 

	The following lemma is well known in the case of locally compact groups.
	\begin{lemma}\label{lem1}
		Let $H$ be a non-discrete ultraspherical hypergroup with Haar measure $\lambda.$ Then $H$ possesses a compact nowhere dense subset of positive Haar measure.
	\end{lemma}
	\begin{proof}
		Suppose that $H$ is second countable. In particular, $H$ is separable. Let $\{t_n\}$ be a countable dense subset of $H.$ Now, for each $n\in\mathbb{N},$ choose an open set $V_n$ such that $\dot{e}\in V_n$ and $\lambda(t_n*V_n)\leq \frac{1}{2^{n+1}}.$ Let $E=\left(\underset{n\in\mathbb{N}}{\cup}t_n*V_n\right)^c.$ Then $E$ satisfies the requirements of the lemma.
		
		Now, suppose that $H$ is not second countable. Using \cite[Theorem 1.1]{kum17}, there exists a compact subhypergroup $K$ of $H$ such that $H//K$ is second countable. Now, by the preceding paragraph, there exists a subset $\widetilde{E}$ of $H//K$ such that $\widetilde{E}$ is nowhere dense and has positive Haar measure. Let $E=\textsf{q}^{-1}(\widetilde{E}),$ where $\textsf{q}$ denotes the canonical quotient map from $H$ onto $H//K.$ This set $E$ satisfies the requirements of the lemma. 
	\end{proof}
	
	Let $G$ be a locally compact group. The coset ring, denoted $\sR(G),$ is defined as the Boolean ring generated by all cosets of subgroups of $G.$ The closed coset ring, denoted $\sR_c(G),$ is defined to be $$\sR_c(G)=\{E\in \sR(G):E~{\text {is closed in} }~ G \}.$$ Let $ H $ be an ultraspherical hypergroup associated with $ G $. We will write $$\sR(H) = \{\dot E \subseteq H: p^{-1}(\dot E)\in \sR(G)\},$$ and $$\sR_c(H)=\{\dot E\subseteq H: p^{-1}(\dot E) \in  \sR_c(G) \}.$$
	
	\subsection{Fourier Algebra} 
	Let $B(H)$ denote the Fourier-steiltjes algebra on $ H $. Then $$B(H)=B_\pi(G)=\{u \in B(G): u ~ is~  \pi \text{-} radial\}.$$ Let $C^*(H)$ denote the hypergroup C*-algebra associated with $H.$ Then $B(H)=C^*(H)^*.$ For each $ u \in B(H) $, we denote by $ \tilde u $ the element in $B_\pi(G)$ associated with $ u $; that is, $ u(\dot x) = \tilde u(x) $ for all $ x \in G $ and $ \Vert u \Vert _{B(H)}=\Vert \tilde u \Vert _{B(G)}$.
	
	Let $A(H)$ denote the Fourier algebra corresponding to the  ultraspherical hypergroup
	$H$.  Recall that the Fourier algebra
	$A(H)$ is semisimple, regular and Tauberian \cite[Theorem 3.13]{mur08}. As in the group case, let $\lambda$ denote the left regular representation of $H$ on $L^2(H)$ given by 
	$$
	\lambda(\dot{x})(f)(\dot{y})=f(\check{\dot{x}}\ast\dot{y})\quad (\dot{x},\dot{y}\in H, f\in L^2(H)).
	$$
	This can be extended to $L^1(H)$ by $\lambda(f)(g)=f*g$ for all $f\in L^1(H)$ and $g\in L^2(H)$. Let $C^*_{\lambda}(H)$ denote the completion of $\lambda(L^1(H))$
	in $B(L^2(H))$ which is called the reduced $C^*$-algebra of $H$. The von Neumann algebra generated
	by $\{\lambda(\dot{x}): \dot{x}\in H\}$ is called the von Neumann algebra of $H$, and is denoted by $VN(H)$. Note that $VN(H)$ is isometrically isomorphic to the dual of $A(H)$. Moreover, $A(H)$ can be considered as an ideal of $B_\lambda(H)$, where $B_\lambda(H)$ is the dual of $C_\lambda^*(H)$. As an immediate consequence of \cite[Theorem 4.2]{mur08} and \cite[Lemma 3.7]{dkl14} we obtain the following
lemma.
	
	\begin{lemma} \label{lem12}
		Suppose that $H$ is an ultraspherical hypergroup associated to an amenable locally compact group $G$. Then
		\item[(i)]
		$ C^*_\lambda(H)= C^*(H) $ and $ B_\lambda(H) = B(H) $. 
		\item[(ii)]
		The Fourier algebra $ A(H) $ admits a bounded approximate identity. Let $ (e_\alpha) $ be a bounded approximate identity of $ A(H) $. Then, after passing to a subnet if necessary, $e_\alpha \rightarrow 1 $ in the weak$ ^* $-topology $ \sigma(B_\lambda (H) , C^*_\lambda(H)) $.
	\end{lemma}
	
	\begin{lemma} \label{lem2}
		Let $ H $ be an ultraspherical hypergroup. Then $ C^*(H) $ is a Banach $B(H)$-bimodule with the module action given by
		$$\langle\varphi \cdot \lambda(f) , \psi\rangle
		=\langle \lambda(f) , \varphi \psi\rangle, 
		\qquad (\varphi, \psi \in B(H), f \in L^1(G)).$$
		In particular, pointwise multiplication  in $ B(H) $ is weak$ ^* $-separately continuous.
	\end{lemma}
	\begin{proof}
		For each $ \varphi \in B(H) $ and $ f \in L^1(H) $, we have $$\langle\varphi \cdot \lambda(f) , \psi\rangle
		=\langle \lambda(f) , \varphi \psi\rangle =\int_H f(\dot x) \varphi(\dot x) \psi(\dot x) d \dot x =\langle \lambda( \varphi f) , \psi\rangle.$$ Thus $ \varphi \cdot \lambda(f)  \in C^*(H). $
		Now, the conclusion follows from the density of $L^1(H)$ inside $C^*(H)$ 
	\end{proof}
	
	\subsection{Commutative Banach algebras} Let $\mA$  be a regular, semisimple, commutative Banach algebra with the Gelfand structure
	space $ \Delta(\mA) $. For a closed ideal $I$ of $\mA$, the zero set of $I$, denoted $Z(I )$, is a closed subset of  $ \Delta(\mA) $
	given by
	$$Z(I ) =\{x \in \Delta(\mA): \widehat a(x)=0 ~{\text{for all } a \in I}\}. $$
	Associated to each closed subset $E$ of  $\Delta(\mA),$  we
	define the following  distinguished ideals  with zero set equal to $E$:
	\begin{align*}
		k_\mA(E)
		&=\{a \in \mA:  \widehat a (x)=0 ~ {\text{for all } x \in E} \},\\
		j_\mA(E)
		&=\{a \in \mA: \widehat a ~ \text{has compact support disjoint from}~E \},\\
		J_\mA(E)
		&=\overline{j_\mA(E)}.
	\end{align*}
	Note that, for each ideal $ I $ of  $ \mA $ with zero set $ E $, we have $j_\mA(E) \subseteq I \subseteq  k_\mA(E) $(cf. Kaniuth \cite[Theorem 5.1.6]{kan09}).
	Recall that a closed subset $E$ of $ \Delta(\mA) $ is a {\it set of  synthesis} or {\it spectral set } if
	$k_\mA(E) = J_\mA(E)  $;  that is, $ k_\mA(E) $ is the only closed ideal with zero set  equal to $E$.
	
	\begin{definition}
		An element $a$ of a Banach algebra $\mA$ is said to be {\it power
			bounded} if 
		$ \sup_{n \in \mathbb{N}} \Vert a^n\Vert<\infty $.
		The Banach algebra $\mA  $ is said to have the {\it power boundedness property} if  any $ a\in A $ with $ r(a) =\inf_{n \in \mathbb{N}} \Vert a^n\Vert^{\frac{1}{n}}\leq1,$ is power bounded.
	\end{definition}
	 
	It is clear that if $ a \in \mA $ is power bounded, then $ r(a) \leq 1. $

	\section{Power bounded elements and Coset ring}
	In this section, we study two sets (defined below) associated to an element of $B(H).$ This section is motivated by some results in \cite{klu10,klu11,mus19}.
	
	\subsection{Coset ring and power bounded elements}
	 We shall begin by defining those two sets.
	\begin{definition}
		Let $ u \in B(H) $.  Let   $  \dot E_u$ and $ \dot F_u $ denote the sets in $ H $ associated to $ u $ given by
		$$ \dot E_u=\{\dot x\in H: \vert u(\dot x)\vert = 1\} \qquad \text {and} \qquad \dot F_u=  \{\dot x\in H:  u(\dot x) = 1\}.$$ Also, associated to $ u $ are the sets $ E_u $ and $ F_u $ in $ G $ defined by $$ E_u=p^{-1}(\dot E_u ) \qquad \text {and} \qquad  F_u= p^{-1}( \dot F_u).$$
	\end{definition}
	
	\begin{rem} \label{rem1}
		Let $ u \in B(H) $ and $ \tilde u $ be the element in $ B(G)  $ associated with $ u $. Then $$\{ x \in G: \tilde u  (x)=1 \} = p^{-1}( \dot F_u)= F_u,$$ and $$\{ x \in G: \vert \tilde  u  (x)\vert=1 \} = p^{-1}( \dot E_u)= E_u.$$
	\end{rem}
	The following theorem generalizes  \cite [Theorem 4.1]{klu10} to the context of ultraspherical hypergroups. Also, the proof of this is a direct consequence of Remark \ref{rem1} and \cite[Theorem 4.1]{klu10}.
	
	\begin{lemma} \label{thm3}
		Let $ H $ be any locally compact ultraspherical hypergroup and let  $ u \in B(H) $ be power bounded. Then the sets $ \dot E_u $ and $\dot F_u $ are in 
		$ \sR_c(H) $.
	\end{lemma}
	
	The proof of our next result follows the same lines as in  \cite[Proposition 3.6]{dkl14} and hence we omit it.
	\begin{lemma} \label{lem8}
		Let $ K $ be a closed subhypergroup of an unltraspherical hypergroup $ H $. Then for any $ u \in B(K) $, there exists $ v \in B(H)$   such that $ v|_K = u $ and $ \Vert v \Vert_{B(H)} = \Vert u \Vert_{B(K)} $.
	\end{lemma}
	Our next lemma says that we cannot replace in Lemma \ref{thm3}, the weaker condition that $\|u\|=1.$ As the proof of this follows similar to \cite[Lemma 4.3]{klu10}, we omit it.
	
	\begin{lemma} \label{lem6}
		Let $ \dot E $ be a closed subset of $ H $ such that $ H\setminus \dot E $ is $ \sigma $-compact. Then there
		exists $ u \in B(H) $ with $ \Vert u\Vert_\infty =1 $ and $ \dot F_u=\dot E $.
	\end{lemma}
	
	Here is an analogue of the Host's idempotent theorem.
	\begin{lemma} \label{lem7}
		Let $u$ be a complex-valued functions on $H$. Then    $ u $ is an idempotent of $ B(H) $  if and only if $ u=\1_{\dot E} $ for some $ \dot E \in \sR_c(H) $.
	\end{lemma}
	\begin{proof}
		Let $u$ be an idempotent element of $B(H).$ Then $\tilde u$ is an idempotent element of $B(G)$ and hence by Host's idempotent theorem \cite{hos86}, there exists $E \in \sR_c(G)$ such that $\tilde u = \1_E.$ Now, it is plain to show that $ E={p^{-1}(\dot E_u)}$. Hence $\dot E_u\in \sR_c(H)$ and $u=\1_{\dot E_u}$. 
		
		For the converse, let  $ u=\1_{\dot E} $ for some $ \dot E \in \sR_c(H) $. Then, by Host's idempotent theorem \cite{hos86}, $ v=\1_{p^{-1}(\dot E)} $ is an idempotent in $B(G)$. It is clear that $v $ is constant on each orbit $\mathcal{O}_x$ and hence $v \in B_\pi(G) $. Now, since for each $x \in G$ we have $v(x) = u(\dot x)$, thus $u$ is an idempotent in $B(H)$.
	\end{proof}
	
	\begin{rem} \label{rem2}
		Since elements of $B(H)$ are continuous functions, it is clear that an idempotent of $B(H)$ has to be of the form $\1_{\dot F}$ where $\dot F$ is a clopen subset of $ H.$
	\end{rem}
	As an immediate consequence of Lemma \ref{lem7} and Lemma \ref{thm3},  we obtain the following corollary.
	
	\begin{corollary}\label{cor3}
		Let $ H $ be an  ultraspherical hypergroup. For any subset $ \dot E $ of $ H $, the following are equivalent.
	
			\item[(i)] $ \dot E \in \sR_c(H) $;
			\item[(ii)] $ \dot E = \dot F_u $ for some idempotent $ u \in B(H) $;
			\item[(iii)] $ \dot E =\dot  F_u $ for some power bounded $ u \in B(H).$
	\end{corollary}
	In the next corollary, we show that the interior and boundary of a set from the coset ring also belongs to the coset ring.
	\begin{corollary}\label{lem9} 
		Let $ H $ be an ultraspherical hypergroup associated to an amenable locally compact group $ G $ and let $ \dot E \in \sR_c(H) $.
			\item[(i)] The interior $ \dot E^\circ $ and the boundary $\partial \dot E $ of $ \dot E$ both belong to $\sR_c(H).$
			\item[(ii)] $\dot E^\circ = \dot F_u  $ for some power bounded $ u \in B(H).$
			\item[(iii)] $ \partial \dot E = \dot F_u $ for some power bounded $ u \in B(H) $.    
	\end{corollary}
	\begin{proof}
		By Corollary \ref{cor3}, we only have to show (i). Since the quotient map $ p:G\rightarrow H $ is open, we have $ p^{-1}(\dot E^\circ) =p^{-1}(\dot E)^\circ $ and $ p^{-1}(\partial\dot E) =\partial p^{-1}( \dot E) $.  Now since  $ p^{-1}(\dot E)  \in \sR_c(G)$,  by \cite [Lemma 1.1]{klu14}, $p^{-1}(\dot E)^\circ  $ and $ \partial p^{-1}( \dot E)  $ both belong to $ \sR_c(G).$ In other words, the sets $\dot E^\circ$ and  $\partial \dot E$ are in $\sR_c(H).$ 
	\end{proof}

	\subsection{On the sequence \texorpdfstring{$\{u^n\}_{n\in\mathbb{N}}$}{}}
	We now study the sequence $\{u^n\}_{n\in\mathbb{N}},$ where $u$ is a power bounded element of $B(H).$
	
	Our first lemma gives the properties of the weak*-limit of the sequence $\{u^n\}_{n\in\mathbb{N}},$ whenever it exists.
	\begin{lemma} \label{lem3}
		Let $H$ be an ultraspherical hypergroup associated to a locally compact group $G$. Let $ u \in B(H) $  and suppose that 
		$\theta =\wsn\ u^n $
		exists. Then 
			\item[(i)] $ \theta $ is an idempotent.
			\item[(ii)] $\theta$ satisfies $ \theta u =\theta$ and $\theta = {\bf 1}_{\dot{F}_{u}^\circ}.$
			\item[(iii)]  the set $\dot{F}_{u}^\circ $ is closed in $G$ and $\lambda_H(\dot{F}_u \setminus \dot{F}_{u}^\circ)=0$. 
	\end{lemma}
	\begin{proof} 
		By Lemma \ref{lem2},  multiplication in $ B (H) $ is  separately weak$ ^* $  continuous and hence for each $ \varrho \in C^* (H) $, we have
		$$\langle u \theta ,\varrho \rangle 
		=\lim_n \langle u^{n+1}  ,\varrho \rangle = \langle  \theta , \varrho \rangle,$$
		i.e., $ u \theta=\theta.$ Further,
		\begin{align*}
			\langle \theta^2 , \varrho \rangle &=\lim_n \lim_m \langle u^n u^m , \varrho \rangle =\lim_n \lim_m \langle u^{n +m }, \varrho \rangle =  \langle \theta, \varrho \rangle,
		\end{align*}
		showing that $\theta$ is an idempotent.  By Lemma \ref{lem7} and Remark \ref{rem2}, there exists a clopen set $\dot{E}\in\sR(H)$ such that $\theta=\1_{\dot{E}}.$ The remaining assertions follows as in \cite[Lemma 3.1]{kau13}		
	\end{proof}
	Our next result gives conditions under which the weak* limit used in Lemma \ref{lem3} holds. As the proof of this follows exactly as in the \cite[Lemma 3.2]{kau13}, we omit the proof.
	\begin{lemma} \label{lem4}
		Let $u$ be a power bounded element of $ B  (H)  $ and suppose that $\lim_{n\rightarrow \infty}\ \Vert u ^{n+1}- u^n\Vert =0 $.  Then $\theta =\wsn\ u^n $
		exists.
	\end{lemma}
	Here is a characterization of the power boundedness of an element $u\in B(H)$ in terms of the set $\dot{F}_u.$ For the corresponding result on locally compact groups, see \cite[Theorem 3.3]{kau13}.
	\begin{theorem}\label{thm1}
		Let $H$ be an ultraspherical hypergroup associated to a locally compact group $G$ and let 
		$ u \in B (H) $ such that $ \underset{n\rightarrow \infty}{\lim}\ \Vert u ^{n+1}- u^n\Vert=0 $.  Then the following are equivalent.
			\item[(i)] $u$ is power bounded.
			\item[(ii)] The set $ \dot{F}_{u}^\circ $ is closed in $H$, belongs to the coset ring $ \sR(H) $ and satisfies $\wsn\ {\bf 1}_{H\setminus{\dot{F}_{u}^\circ}} u^n =0 $. 
	\end{theorem}
	\begin{proof}
            In view of Lemma \ref{lem2}, Lemma \ref{lem3} and Lemma \ref{lem4}, the proof of this follows exactly as in \cite[Theorem 3.3]{kau13}.
	\end{proof}
	The next lemma is a special case of  \cite[Corollary 1.3]{klu10} and hence we omit the proof. 
	
	\begin{lemma} \label{lem5}
		Let $H$ be an ultraspherical hypergroup. Then for any power bounded element $ u \in B(H) $, the element $ \frac{1+u}{2} $ is power bounded and
		$$\lim_{n\rightarrow \infty} \left\Vert \left(\frac{1+u}{2}\right)^{n+1}-\left(\frac{1+u}{2}\right)^{n}\right\Vert=0.$$
	\end{lemma}
	Combining Lemmas \ref{lem4} and  \ref{lem5} and the fact that $ F_u=F_{\frac{1+u}{2}} $ one obtains the following description of the power bounded elements of $ B(H) $.
	
	\begin{corollary} \label{cor1}
		Let $H$ be an ultraspherical hypergroup associated to a locally compact group $G$. Then,  for any power bounded $ u \in B(H) $, $ \theta= \wsn (\frac{1+u}{2}) ^n$ exists and  $  \theta = \1_{\dot{F}_{u}^\circ} $.
	\end{corollary}
	The following is a simple observation.
	\begin{lemma}\label{PBmod}
		Let $H$ be an ultraspherical hypergroup associated a locally compact group $G.$ If $u\in B(H)$ is power bounded, then so is $|u|.$
	\end{lemma}
	\begin{proof}
		Let us first note that $ \Vert u \Vert = \Vert \bar u  \Vert $ for all $ u \in B(H),$ by \cite[Remark 2.9]{mur07}. Let $u\in B(H)$ be power bounded with $\gamma=\sup_n \Vert u^n \Vert.$ Then for each $ n \in \mathbb{N} $, we have $$\Vert \vert u\vert^{2n} \Vert= \Vert  u ^{ n} \bar u ^{ n} \Vert\leq \Vert  u ^{ n}   \Vert \cdot \Vert    \bar u ^{ n} \Vert=\Vert  u ^{ n}   \Vert^2\leq \gamma^2$$ and $$\Vert \vert u \vert^{2n+1} \Vert
		\leq \Vert \vert  u \vert  \Vert \cdot \Vert \vert u\vert^{2n} \Vert \leq  \gamma^2  	\Vert \vert  u \vert \Vert. $$
		Thus, $ \vert u\vert $ is a power bounded element of $ B(H) $.
	\end{proof}

	The next result  is an analogue of \cite[Proposition 3.5]{kau13}.
	\begin{corollary} \label{cor2}
		Let $H$ be an ultraspherical hypergroup associated to a locally compact group $G$. Then, the set $\dot{E}_{u}^\circ$  is closed in $H$ and $\lambda(\dot{E}_u \setminus \dot{E}_{u}^\circ)=0$. 
	\end{corollary}
	\begin{proof}
		This follows from Lemma \ref{lem3}, Lemma \ref{PBmod} and the fact that $ \dot{E}_u = \dot{F}_{\vert u \vert}.$
	\end{proof}
	
	As an immediate consequence of Corollary \ref{cor1} we obtain the following
	lemma.
	
	\begin{lemma} \label{PropTheta}
		Let $ u \in B(H) $ be power bounded and let $ \theta = \wsn_{\rightarrow\infty}(\frac{1+u}{2}) ^n  $. Then we have
			\item[(i)] $ \langle \theta , f \rangle = \int_{\dot{F}_u}f(\dot x) d \dot x $ for all $ f \in L^1(H) $;
			\item[(ii)] $ \theta =0 $ if and only if $\lambda_H (\dot{F_u})=0 $.
	\end{lemma}
	\begin{lemma} \label{lem14}
		Let $ u \in B(H) $ be power bounded. Then the following assertions hold:
			\item[(i)] If  $ \dot E_u = \dot F_u $, then $\theta=\wsn_{\rightarrow\infty}  u  ^n  $ exists and $ \theta =\1_{\dot F_u^\circ} $.
			\item[(ii)] If $ H $ is discrete, then $\theta=\wsn_{\rightarrow\infty}  u  ^n  $ exists if and only if $ \dot E_u = \dot F_u $.
			\item[(iii)] $\wsn_{\rightarrow\infty} \vert u \vert ^n=0$ if and only if $ \dot E_u^\circ=\emptyset. $
			\item[(iv)] If $\wsn_{\rightarrow\infty} \vert u \vert ^n=0$, then $\wsn_{\rightarrow\infty}  u  ^n=0$.
	\end{lemma}
	\begin{proof}
		(i). Assume that $ \dot E_u = \dot F_u $. Then, we can show that the sequence  $ \{u^n \}_{n\in \mathbb{N}}$ converges pointwise to $ \1_{\dot F_u} $ and also by  Corollary \ref{cor2}, $ \lambda (\dot F_u \setminus \dot F_u^\circ)=0 $. Using these and the fact that $ \Vert u \Vert_\infty \leq 1 $, we have
		$$\lim_n u^n(\dot x)f(\dot x) =\1_{\dot F_u^\circ}(\dot x) f(\dot x) \qquad \text{and} \qquad \vert u^n f \vert \leq\vert f\vert,$$
		almost everywhere on $ H $ and for all $ f \in L^1(H) $.  
		Now, the Lebesgue’s dominated convergence theorem yields
		that
		\begin{align*}
			\langle u^n , \lambda(f) \rangle
			&= \int_H u^n f d\lambda \rightarrow \int_H   \1_{\dot F_u^\circ} f d\lambda= \langle \1_{\dot F_u^\circ}   , \lambda(f) \rangle
		\end{align*}
		as $ n\rightarrow \infty $. Here, the last  equality is a consequence of the fact that if $ u \in B(H) $ is power bounded, then $ \1_{\dot F_u^\circ} \in B(H) $. Now, the equality $ \theta =\1_{\dot F_u^\circ} $ follows from a simple approximation argument.
		
		(ii). If $\theta=\wsn_{\rightarrow\infty}  u  ^n  $ exists, then by Lemma \ref{lem3}, $ \theta = \1_{\dot F_u} $. Since $ H $  is discrete,  weak$^*$-topology on $ B(H) $ is stronger than topology of pointwise convergence and thus  we have
		$$\lim _{n\rightarrow\infty}  u  ^n(\dot x) = \1_{\dot F_u} (\dot x),\  \dot x \in H.$$
		Since $\dot F_u \subseteq \dot E_u,$ we only need to show the reversed inclusion. Towards a contradiction,  suppose that $ \dot x \in \dot E_u \setminus \dot F_u $. Then $ \vert u^n(\dot x) \vert=1 $ for all $ n \in \mathbb{N}$ and hence
		$$  \vert u  ^n(\dot x)-\1_{\dot F_u} (\dot x)  \vert=\vert u  ^n(\dot x)  \vert\rightarrow1\neq 0,$$
		as $ n \rightarrow \infty $, a contradiction. The converse follows from (i).
		
		(iii).
		If $ \dot E_u^\circ=\emptyset $, then Corollary \ref{cor2} implies that $ \lambda_H(\dot E) =0 $. Since $ u \in B(H) $ is power bounded, we have $ \vert u(\dot x) \vert_\infty < 1 $ for all $ \dot x \in H \setminus \dot E $. Thus 
		$$\lim_{n\rightarrow \infty} \vert u^n(\dot x)\vert =0, \qquad (\text {a.e. on}~   H).$$
		The rest of the proof runs as in (i).
		
		Assume that $\wsn_{\rightarrow\infty} \vert u \vert ^n=0$. Then we have
		$$\lambda_H(\dot E_u^\circ) \leq \lambda_H(\dot E_u) = \lambda_H(\dot F_{\vert u\vert}) =0,$$ 
		the last equality being a consequence of Lemma \ref{PropTheta}. Hence (iii) holds.
		
		(iv). The same reasoning as in the first part of the proof of  (iii) applies to here.
	\end{proof} 
	
	The assertion of the following corollary is an immediate consequence of Lemma \ref{lem14} and the fact $ \dot E_{uv} = \dot E_{u}\cap \dot E_{v}$ for all power bounded elements $ u,v \in B(H) $.
	\begin{corollary} 
		Let $ u, v $ be elements of $ B(H) $ such that $\wsn_{\rightarrow\infty} \vert u \vert ^n=0$ and $ v $ is power bounded. Then $\wsn_{\rightarrow\infty}  (uv ) ^n=0$.
	\end{corollary}
	Here is the final result of this section. This is about powers of power bounded elements and an element from the ideal $J_{A(H)}(C),$ where $C$ is a closed subset of $H.$ As the proof of this follows exactly as in \cite[Proposition 4.1]{mus19}, we shall omit the proof.
	\begin{theorem} \label{thm4}
		Let $u   $ be a power bounded element of $ B(H) $ and let $ C $ be a closed subset of $ H $. Then the following are equivalent:
			\item[(i)] $ \lim_{n\rightarrow \infty}\left\Vert \frac{1}{n}\sum_{k=0} ^{n-1}\vert u\vert^{2k} v\right\Vert_{A(H)}=0$ , for all $ v \in J_{A(H)}(C) $.
			\item[(ii)] $ \lim_{n\rightarrow \infty}\Vert u^n v\Vert_{A(H)}=0$, for all $ v \in J_{A(H)}(C)$.
	\end{theorem}
	
	\subsection{Ces\`aro boundedness}
	Here we deal with Ces\`aro boundedness. This concept was introduced by Mustafayev \cite{mus19} in order to study certain non-abelian analogues of the ergodic theorem. We shall derive relations to Ces\`aro bounded elements with the sets $E_u$ and $F_u$ as done in the previous sections.
	
	We now begin with the definition of Ces\`aro boundedness.
	\begin{definition}
		Let $\mathcal A$ be a complex Banach algebra  with the unit element $e$. An element  $a \in \mathcal A  $ is said to be Ces\`aro bounded if $$\sup_{n\in \mathbb{N}} \left\Vert \frac{1}{n}\sum_{k=0}^{n-1}a^k\right\Vert <\infty.$$
	\end{definition}
	The following example shows that Ces\`aro boundedness does not imply power boundedness.
	\begin{rem}
		The Assani matrix 
		$T=\begin{pmatrix}
			-1 & 2 \\
			0 & -1 
		\end{pmatrix}$
		is  Ces\`aro bounded, but not power bounded.
	\end{rem}
	
	The following is an analogue of Lemma \ref{thm3}
	\begin{proposition} \label{pro3}
		Let $H$ be an ultraspherical hypergroup associated to a locally compact group $G$.
		If $ u \in B(H) $ is Ces\`aro bounded, then $ \dot F_u \in \sR_c(H) $.
	\end{proposition}
	
	\begin{proof}
		Since $ u \in B(H) $ is Ces\`aro bounded, the corresponding $ \tilde u \in B(G) $ is Ces\`aro bounded. Thus, by \cite[Proposition 2.4]{mus19},  $p^{-1 }(\dot F_u) =F_{\tilde u}   \in \sR_c(G),$ which is equivalent to saying that $ \dot F_u \in \sR_c(H).$
	\end{proof}
	The following is the analogue of Lemma \ref{lem3}.
	\begin{lemma} \label{lem10}
		Let $H$ be an ultraspherical hypergroup associated to a locally compact group $G$. Let $ u \in B(H) $  and suppose that 
		$$\theta =\wsn \frac{1}{n}\sum_{k=0}^{n-1}u^k   $$
		exists. Then
			\item[(i)] $ \theta $ is an idempotent and satisfies $ \theta u =\theta $. More precisely, $   \theta = {\bf 1}_{\dot F_{u}^\circ}.$
			\item[(ii)] The set $ F_{u}^\circ $ is closed in $H$ and $\lambda(F_u \setminus F_{u}^\circ)=0$.
	\end{lemma}
	\begin{proof}
		First note that
		\begin{align*}
			\theta u &= \wsn \frac{1}{n}\sum_{k=1}^{n}u^k  \\
			&= \wsn \left[\frac{1}{n}\sum_{k=0}^{n}u^k -  \frac{1}{n}\cdot 1\right]\\
			&=\wsn \left[\frac{n+1}{n}(\frac{1}{n+1}\sum_{k=0}^{n}u^k)- \frac{1}{n}\cdot 1\right]  \\
			&=\wsn \frac{1}{n+1}\sum_{k=0}^{n}u^k = \theta.
		\end{align*}
		By Lemma \ref{lem2}  multiplication in $B(H)$ is  separately weak* continuous and therefore $\theta=\theta^2,$ i.e., $\theta$ is an idempotent in $B(H).$ So, by Lemma \ref{lem7} and Remark \ref{rem2} there is clopen subset $\dot{E}$ of $H$ in $\sR_c(H)$ such that $\theta=\1_{\dot{E}}.$ Now, the remaining can be proved by using the same arguments as those employed in the proof of Lemma \ref{lem3}.
	\end{proof}
	
	Our next result is an analogue of Corollary \ref{cor1}.
	\begin{proposition}\label{pro5}
		Let $H$ be an ultraspherical hypergroup associated to a locally compact group $G$. Let $ u $ be a Ces\`aro bounded element of $ B(H) $. Then    
		$$\1_{\dot F^\circ_u} =\wsn \frac{1}{n}\sum_{k=0}^{n-1}u^k.   $$
	\end{proposition}
	
	\begin{proof}
		Let $ u $ be a Ces\`aro bounded element of $ B(H) $. Then for each $ f \in L^1(H) $,  the sequence $ (\frac{1}{n}\sum_{k=0}^{n-1}u^k )_n$ converges pointwise to $ \1_{\dot F_u} f$ as $ n\rightarrow \infty $ and is also dominated $\vert f \vert.$ Hence by the dominated convergence theorem
		\begin{align*}
			\langle \psi_n , \lambda (f) \rangle
			&= \int_H \psi_n(\dot x) f(\dot x) d\dot x
			\rightarrow  \int_H \1_{\dot F_u}(\dot x) f(\dot x) d\dot x\\
			&=\int_H \1_{\dot F^\circ_u}(\dot x) f(\dot x) d\dot x + \int_H \1_{ \partial \dot F_u}(\dot x) f(\dot x) d\dot x \\
			&=\int_H \1_{\dot F^\circ_u}(\dot x) f(\dot x) d\dot x   \qquad (\text{since}~\lambda(F_u \setminus F_{u}^\circ)=0)\\
			&= \langle \1_{\dot F^\circ_u} , \lambda(f)  \rangle.
		\end{align*}
		
		Now, using the density  of $ \lambda (L^1(H)) $ in $ C^*(H) $ and the boundedness of the  sequence $ (\frac{1}{n}\sum_{k=0}^{n-1}u^k )_n,$ the conclusion follows.
	\end{proof}
	
	\begin{proposition} \label{pro6}
		Let $H$ be an ultraspherical hypergroup and let $ u $ be a Ces\`aro bounded element of $ B(H) $ such that $\lim_n \frac{1}{n} \Vert u^n v \Vert_{A(H)}=0 $ for all $ v \in A(H) $. Consider the following statements:
			\item[(i)] $\theta_v = \lim_n\frac{1}{n}\sum_{k=0}^{n-1}u^k v$ exists in $ A(H)   $-norm topology for all $ v \in A(H).$
			\item[(ii)] $ \dot F_u $ is an open subset of $H.$
			\item[(iii)] $ \theta_v = \1_{\dot F_u} v $ and $\Vert  \1_{\dot F_u}  v \Vert _{A(H)}= dist ( \overline{(1-u) A(H)}, v)$ for all $ v \in A(H).$
   
		Then the following assertions hold: $(i)\Longleftrightarrow (ii)\Longrightarrow(iii).$ Furthermore, if u is Ces\`aro contractive, then (iii) implies the following:
		\item[(iv)]  $\Vert  \1_{\dot F_u}  v \Vert _{A(H)}= dist ( \overline{(1-u) A(H)} , v)$ for all $ v \in A(H) $.
	\end{proposition}
	\begin{proof} 
		The proof of this is exactly the same as in \cite [Theorem 2.8]{mus19}
	\end{proof}
	
	\begin{proposition} \label{pro4 }
		Let $H$ be an ultraspherical hypergroup such that the underlying locally compact group $G$ is amenable. Then for any Ces\`aro bounded element
		$ u \in B(H) $, we have the following:
		\item[(i)]
		The sets $ \overline{(1-u) A(H)} $ and $ I(\dot F_u) $ are equal.
		
		\item[(ii)]
		The ideal $ \overline{(1-u) A(H)} $ has a bounded approximate identity.
		
		\item[(iii)]
		Let $ (e_\alpha) $ be a bounded approximate identity of $\overline{(1-u) A(H)}  $. Then, after passing to a subnet if necessary, $e_\alpha \rightarrow \1_{H\setminus \dot F_u^\circ} $ in the weak$ ^* $-topology $ \sigma(B_\lambda (H) , C^*_\lambda(H)) $.
		\item[(iv)]
		There exists an idempotent $ \theta \in B(H) $ such that 
		$ \overline{(1-u) A(H)}^{weak^*}= \theta  B(H) $. More precisely, $ \theta = \1_{H\setminus \dot F_u^\circ} $.
		
		Moreover, If $ \dot F_u $ is an open subset of $ H $, then
		\item[(v)] $ \overline{(1-u) A(H)}=I(\dot F_u) = \1_{H\setminus \dot F_u}  A(H) $.
	\end{proposition}
	\begin{proof}
		(i). Suppose that $ u \in B(H) $ is Ce\`saro bounded. Then, by Proposition \ref{pro3} and Lemma \ref{lem11}, $ \dot F_u  $  is a set of synthesis  for $ A(H) $. Now, since $ Z( \overline{(1-u) A(H)})= \dot F_u   $, the equality of $ \overline{(1-u) A(H)} $ and $ I(\dot F_u) $ holds.
		
		(ii). This follows from (i) and Lemma \ref{lem11}.
		
		(iii). 
		Let $ (e_\alpha) $ be a bounded approximate identity of $\overline{(1-u) A(H)}  $. Passing to a subnet if necessary, we can assume that $ e_\alpha \rightarrow \theta $ in the weak$ ^* $-topology $ \sigma(B_\lambda (H) , C^*_\lambda(H)) $ for some $ \theta \in B(H) $. Let $ f \in  C_c(H) $. Then 
		\begin{align*}
			\langle e_\alpha , \lambda (f)\rangle 
			&= \int_H e_\alpha(\dot x) f(\dot x) d \dot x\\
			&= \int_{H \setminus \dot F_u} e_\alpha(\dot x) f(\dot x) d \dot x +
			\int_{\dot F _u} e_\alpha(\dot x) f(\dot x) d \dot x 
			\\
			&= \int_{\dot F \setminus  \dot F_u^\circ} e_\alpha(\dot x) f(\dot x) d \dot x +\int_{H \setminus \dot F_u^\circ} e_\alpha(\dot x) f(\dot x) d \dot x   
			\qquad (\text {since}~ e_\alpha|_{\dot F_u}=0)\\
			&= \int_{H \setminus \dot F_u^\circ} e_\alpha(\dot x) f(\dot x) d \dot x
			\qquad (\text {since by Lemma \ref{lem10}},~ \lambda (\dot F_u \setminus \dot F_u^\circ)=0)\\
			&= \int_{H \setminus \dot F_u^\circ} e_\alpha(\dot x) v (\dot x) f(\dot x) d \dot x
			\qquad (v \in I(\dot F_u)  ~\text{with}~ v\mid _{\supp(f)\cap (H \setminus \dot F_u) }=1)\\
			&= \int_H e_\alpha(\dot x) v(\dot x) \1_{H \setminus \dot F_u^\circ}(\dot x) f(\dot x) d \dot x\\
			&= \langle e_\alpha v , \lambda(\1_{H \setminus \dot F_u^\circ} f) \rangle
			\rightarrow \langle  v , \lambda(\1_{H \setminus \dot F_u^\circ} f) \rangle \\
			&= \langle \1_{H \setminus \dot F_u^\circ} , \lambda( f) \rangle, 
		\end{align*}
		where the last equality holds because
		$$  H\setminus  \dot F_u^\circ = (H\setminus \dot F_u )\cup (\dot F_u\setminus  \dot F_u^\circ) \quad \text {and } \quad \lambda (\dot F_u \setminus \dot F_u^\circ)=0.$$
		Now, using the density  of $ \lambda (C_c(H)) $ in $ C^*(H) $ and boundedness of the net $ (e_\alpha) $, we have 
		$$\wsa e_\alpha = \1_{H \setminus \dot F_u^\circ}.$$
		
		(iv). The statement follows from (iii) and the fact that  $\1_{H\setminus \dot F_u^\circ}   $ is an identity for $ \overline{(1-u) A(H)}^{weak^*} $ in $ B(H)$.
		
		(v). If $ \dot F_u  $ is open,  then,  by (iv), $\1_{H\setminus \dot F_u}  $ is an idempotent in $ B(H) $. Thus $ \1_{H\setminus \dot F_u}  A(H) $ is closed ideal of $ A(H) $. Let $ (e_\alpha) $ be a bounded approximate identity of $ A(H) $. Then $ (\1_{H\setminus \dot F_u} e_\alpha) $ is a bounded approximate identity for $ I(\dot F_u) $. By (i), $ \overline{(1-u) A(H)}=I(\dot F_u)$, thus for each $ v \in  \overline{(1-u) A(H)}$, we have
		$$v = \lim_\alpha \1_{H\setminus \dot F_u} e_\alpha v.$$
		Hence, $ v \in \overline{\1_{H\setminus \dot F_u}  A(H)}= \1_{H\setminus \dot F_u}  A(H) $. The reverse inclusion is clear.
	\end{proof}
	The following proposition gives a description the set $F_u,$ where $u$ is a Ces\`aro bounded.
	\begin{proposition} \label{pro7 }
		Let $H$ be an ultraspherical hypergroup. Let $ u $ be a Ces\`aro bounded element of $ B(H) $ satisfying $\lim_n \frac{1}{n} \Vert u^n v \Vert_{A(H)}=0 $ for all $ v \in A(H) $ and $ \dot F_u $ is an open subset of $ H $, then 
		$$ \overline{(1-u) A(H)} = \1_{H\setminus \dot F_u}  A(H). $$
	\end{proposition}
	\begin{proof}
		The proof of this is exactly same as \cite[Corollary 2.10]{mus19}.
	\end{proof}
	
	\begin{corollary}\label{pro5 }
	    Let $H$ be an ultraspherical hypergroup associated to an amenable locally compact group $G$. Then for any power bounded element
		$ u \in B(H) $, we have the following:
		\item[(i)]
		The sets $ \overline{(1-u) A(H)} $ and $ I(\dot F_u) $ are equal.
		\item[(ii)]
		Let $A_0(u)= \{v \in A(H): \lim_n \Vert u^n v\Vert=0\}$. Then  $A_0(u)   $ and $ I(\dot E_u) $ are equal.
		\item[(iii)]
		The ideal $ \overline{(1-u) B(H)} $ has a bounded approximate identity.
		\item[(iv)]
		There exists an idempotent $ \theta \in B(H) $ such that 
		$ \overline{(1-u) B(H)}^{weak^*}= \theta  B(H) $. 
	\end{corollary} 
	\begin{proof}
	    As powerboundedness implies Ces\`{a}ro boundedness, this corollary is an immediate consequence of Proposition \ref{pro4 }.
	\end{proof}
	
	\begin{definition}
		Let $ H $ be a discrete ultraspherical hypergroup. For any subset $ \dot E $ of $ C^*_\lambda (H) $, we write
		$$C^*_\lambda (E)= \overline{ \text {span}\{\lambda(\dot x) : \dot x \in \dot E\}}^{\Vert . \Vert _{C^*_\lambda (H) }}$$
	\end{definition}
	The following lemma generalises Lemma 4.3 of \cite{klu14} to the context of ultraspherical hypergroups.
	As the proof of this is same as the proof of \cite[Lemma 4.3]{klu14} we omit the proof.
	\begin{lemma}
		Let $ H $ be a discrete ultraspherical hypergroup. 
		If $ E \in \sR(H) $, then $ \1_{\dot E } \cdot C^*_\lambda (H) = C^*_\lambda (E) $.
	\end{lemma}
	
	\begin{corollary}
		Let $  H$ be a discrete ultraspherical hypergroup, $\dot E\in\sR_c(H)$ and $u \in B_\lambda (H)$.
		Then the following are equivalent.
		\item[(i)] 
		$ u^n \1_{\dot E}\rightarrow 0 $ in weak$ ^* $ topology of $B_\lambda (H)  $.
		\item[(ii)] 
		For each $ T \in C^*_\lambda (\dot E), \lim_{n\rightarrow \infty} \langle u^n , T \rangle = 0$.
	\end{corollary}

	\section{Power boundedness of \texorpdfstring{$A(H)$}{A(H)} and \texorpdfstring{$B(H)$}{B(H)}}
	In this section, we study the power boundedness of the Fourier algebra $A(H)$ and Fourier-Stieltjes algebra $B(H).$
	
	\subsection{Power bounded  property of \texorpdfstring{$A(H)$}{A(H)}}
		We shall begin with a simple lemma.
	
	\begin{lemma}\label{func1}
		Let $U$ be an open subset of $H.$ Then there exists $u\in A(H)$ and an open set $V$ of $H$ such that the following holds:
			\item[(i)] $0\leq u\leq 1$.
			\item[(ii)] $\|u\|_{A(H)}=1=u(\dot{e})$.
			\item[(iii)] $\supp(u)\subset U$.
			\item[(iv)] $u(\dot{x})>0$  for all $\dot{x}\in V$.
	\end{lemma}
	\begin{proof}
		There exists a symmetric, relatively compact neighbourhood $V_1$ of $\dot e$ in $H$ such that $V_1^2\subset U.$ Let $u= \dfrac{1}{\lambda(V_1)} \1_{V_1} \ast \1_{V_1} $ and $V=\{\dot x \in H: ~ u(\dot x)>0\}.$ Then $u$ and $V$ satisfy the requirements of the lemma.
	\end{proof}
	\begin{lemma}\label{func2}
		Let $H$ be a second countable ultraspherical hypergroup and let $K$ a closed subset of $H.$ Then there exists $u\in A(H)$ such that 
			\item[(i)] $0\leq u\leq 1$ and
			\item[(ii)] $K=u^{-1}(0)$.
	\end{lemma}
	\begin{proof}
		Let $\mathcal{V}$ be a neighbourhood basis of $ \dot e $ in $H$ such that for each $V\in \mathcal{V},$ there exists $u_V\in A(H)$ with the property that $0\leq u_V\leq 1,$ $\|u_V\|_{A(H)}=1=u(\dot{e}),$ $\supp(u_V)\subset V$ and $u_V(\dot{x})>0$  for all  $ \dot{x}\in V$ (such neighbourhoods exists by the previous lemma). Let $K$ be any closed set in $H$ and $W=H\setminus K.$ Since $H$ is second countable, there are sequences $\{\dot a_n\}_{n\in \mathbb{N}} \subset W$  and $\{V_n\}_{n\in \mathbb{N}} \subset \mathcal{V}$ such that $W=\bigcup_{n \in \mathbb{N}} \dot a_n \ast V_n.$ For each $n\in\mathbb{N},$ let $u_n=u_{V_n}$ and let $ u= \sum_1^{\infty}2^{-n}L_{\dot a_n} u_n $ where 
		$ L_{\dot a}f(\dot x):= f(\check{\dot{a}}\ast  \dot x) $. Then $u\in A(H),$ since $ \Vert L_{\dot a_n} u_n \Vert \leq \Vert u_n\Vert = 1.$ By \cite[Lemma  4.1B]{jew75} it is obvious that $ u(\dot x)>0 $ for all $ \dot x \in W $ and $ u(\dot x)=0 $ for all $ \dot x \in K.$
	\end{proof}
	
	\begin{theorem}\label{CPBP}
		Let $ H $ be an ultraspherical hypergroup. The Banach algebra $A(H)$ has the  power boundedness  property if and only if $H$ is discrete.
	\end{theorem}
	\begin{proof}
		If  $H$ is discrete, then the power boundedness of $A(H)$ is a consequence of \cite[Corollary 2.3] {sch70}. Thus, we are left with the proof of the forward part. Towards a contradiction, assume that $ H $ is nondiscrete. 
		
		Suppose that $H$ is second countable. In view of Lemma \ref{lem1} and Lemma \ref{func2} the proof now follows exactly as in \cite[Proposition 1]{gra76} .

		Now let $H$ be an ultraspherical hypergroup which is not second countable. Let $U$ be a symmetric relatively compact neighbourhood of $\dot{e}$ and let $H^\prime=\underset{n\in\mathbb{N}} \bigcup{U^n},$ where $U^n=U*U*\ldots *U\ (n$-times$).$ It is clear that $H^\prime$ is an open subhypergroup of $H$ and hence closed. Since $A(H)$ has the power bounded property, it follows from \cite[Lemma 5.1]{kum14} that $A(H^\prime)$ also has the power bounded property. Since $H$ is non-discrete, so is $H^\prime.$ Choose a sequence $\{U_n\}_{n \in \mathbb{N}}$ of relatively compact neighbourhoods of $\dot{e}$ such that $\lambda(U_n)\rightarrow 0.$ By \cite[Theorem 1.1]{kum17}, there exists a compact subhypergroup $H^{\prime\prime}\subseteq\underset{n\in\mathbb{N}}{\cap}U_n$ such that $H^\prime//H^{\prime\prime}$ is a second countable ultraspherical hypergroup. Since $H^\prime//H^{\prime\prime}$ is also an ultraspherical hypergroup based on the subgroup $p^{-1}(H^\prime)$ and the spherical projector $\pi|_{C_c(p^{-1}(H^\prime))},$ it follows that $A(H^\prime//H^{\prime\prime})$ can be identified isometrically inside $A(H^\prime).$ As $A(H^\prime)$ has the power bounded property, it follows that $A(H^\prime//H^{\prime\prime})$ also has the power bounded property. Therefore, by the preceding paragraph it follows that $H^\prime//H^{\prime\prime}$ is discrete, which forces us to conclude that $H^{\prime\prime}$ is open and hence $\lambda(H^{\prime\prime})>0,$ a contradiction. 
	\end{proof}
	
	\subsection{Power bounded property for \texorpdfstring{$B(H)$}{B(H)}}
	Our  main aim of this part is to show that $B(H)$ is power bounded if and only if $H$ is finite. Our idea is to make use of weakly almost periodic functions on $H$ and weakly almost periodic functionals on $L^1(H).$ We shall begin this section by defining the same.
	\begin{definition}\mbox{ }
		\begin{enumerate}[(i)]
			\item Let $H$ be an arbitrary locally compact hypergroup. A function $f \in C_b(H)$ is called weakly almost periodic if the left orbit $O_L(f) = \{\ell_xf: x \in H\},$ where 
			$\ell_xf(t)= f(x \ast t)~ (t \in H),$ is relatively weakly compact in $C_b(H)$. We denote the set of all weakly almost periodic functions on $H$ by $wap(H).$
			\item Let $H$ be a locally compact hypergroup with a Haar measure  $\lambda$. Let $f \in L^\infty (H)$ and
			$\mathcal{O}_L (f) = \{a \cdot f: a \in L^1(H), \Vert   a \Vert \leq 1  \},$ where $a \cdot f=\frac{1}{\Delta} \check a \ast f~ ( a \in L^1(H))$. The functional $f$ is called weakly almost periodic on $L^1(H)$ if $\mathcal{O}_L (f)$ is $\sigma(L^\infty (H) ,L^\infty (H)^* )$-relatively compact. We denote the set of all weakly almost periodic functionals on $L^1(H)$ by $WAP(L^1(H))$.
		\end{enumerate}
	\end{definition}
		
	\begin{lemma}
		Let $H$ be a locally compact hypergroup with a Haar measure  $\lambda$. Let $f \in L^\infty (H)$ which is in $WAP(L^1(H))$. Then $f$ is $\lambda$-a.e continuous on $H$.
	\end{lemma}
	
	\begin{proof}
		Let $f \in WAP(L^1(H))$ and $(e_\alpha)_\alpha$ be an approximate identity for $L^1(H)$ with $\Vert e_\alpha \Vert \leq 1$ for all $\alpha$. As $f\in WAP(L^1(H))$ and also as the net $(e_\alpha)$ converges to a right identity of $L^\infty (H)^*,$ after passing to a subnet if necessary, we can assume that $\wa e_\alpha \cdot f = f.$ Now, since $e_\alpha \cdot f \in C_b(H)$ for all $\alpha$, we conclude that $f \in \overline{C(H)}^{weak} = C(H).$
	\end{proof}

	\begin{lemma}\label{lem15}
		Let $H$ be a locally compact hypergroup with a Haar measure  $\lambda$ and let $f \in C_b(H)$. Then $f \in wap(H)$ if and only if the set $\mathbb O_L(f) =\{f \cdot \mu : \mu \in M(H), \Vert \mu\Vert \leq 1 \}$ is relatively weakly compact in $C_b(H)$.
	\end{lemma}
	\begin{proof}
		Assume that $\mathbb O_L(f)$ is relatively weakly compact.  Since $\mathcal O_L(f) \subset \mathbb O_L(f)$, we have $f \in wap(H)$. For the converse, let $\mathcal O_L(f)$ be  relatively weakly compact in $C_b(H)$. Then, by Krein-\v{S}mulian theorem \cite[Theorem 2.8.14]{meg12},  the closure of $co (O_L(f))$ in $C_b(H)$ is weakly compact. Let $\mu \in M(H)$ with $\Vert \mu \Vert \leq 1$ such that $C=supp(\mu)$ is compact. By \cite[p.71, Corollary 2]{bou65}, there exists a net $(\mu_\alpha)_\alpha$ in $co\{\delta _x : x \in C\}$ such that 
		$$\mu = \wsa\ \mu_\alpha. $$
		Now, since the net
		$(f \cdot \mu_\alpha )_\alpha$ is contained in $\overline{co}(O_L(f))$ and as $\overline{co}(O_L(f))$ is weakly compact, by passing to a subnet if necessary, we can assume that 
		$f \cdot \mu_\alpha   \rightarrow h$ for some $h \in C_B(H)$ in the weak topology. Thus for each $t \in H$, we have
		\begin{align*}
			h(t) = \langle \delta_t , h \rangle 
			&= \lim_\alpha 
			\langle  \delta_t ,  f \cdot  \mu_\alpha   \rangle\\
			&= \lim_\alpha 
			\langle  \mu_\alpha \ast \delta_t ,  f    \rangle\\
			&= 
			\langle  \mu \ast \delta_t ,  f    \rangle\\
			&= 
			\langle  \delta_t ,  f \cdot \mu   \rangle.
		\end{align*}
		Hence, $f \cdot \mu = h \in \overline{co}(O_L(f))$. Now, let $\mu$ be an arbitrary measure in the unit ball of $M(H) $. Since the set of compactly supported measures is norm-dense in $M(H)$, there is a sequence  $\{\mu_n\}_{n\in \mathbb{N}}$ of compactly supported measures in $M(H)$ such that 
		$\lim_n\Vert \mu_n - \mu \Vert =0$. Thus
		$$\Vert f \cdot \mu_n -f\cdot  \mu \Vert_\infty \leq \Vert f\Vert_\infty \Vert \mu_n - \mu \Vert \rightarrow 0, $$
		Hence, $f \cdot \mu \in  \overline{co}(O_L(f))$ and therefore,
		$\mathbb O_L(f) \subseteq \overline{co}(O_L(f)),$ which implies that it is relatively weakly compact.
	\end{proof}

	\begin{proposition}
		Let $H$ be a locally compact hypergroup with a Haar measure  $\lambda$. Then $WAP(L^1(H))= wap(H)$. 
	\end{proposition}

	\begin{proof}
		Assume that $f \in WAP(L^1(H))$ and $(e_\alpha)_\alpha$ be an approximate identity of $L^1(H)$ of bound 1. Let $\mu $ be an arbitrary measure in the unit ball of $ M(H)$. Then, since $f \in WAP(L^1(H))$, After passing to a subnet if necessary, we can assume that $f \cdot \mu \ast e_\alpha \rightarrow g$ in the $\sigma(L^\infty (H) , L^\infty (H)^*)$-topology for some $g \in C_b(H)$. Now, since  $\sigma(L^\infty (H) , L^\infty (H)^*)$-topology on $C_b(H)$ is stronger than $\sigma(C_b(H), M(H))$-topology, we have
		$$\lim_\alpha \langle f \cdot \mu \ast e_\alpha , \nu\rangle
		=\lim_\alpha \langle f , \mu \ast e_\alpha \ast \nu\rangle
		= \langle f \cdot \mu , \nu\rangle $$
		for all $\nu \in M(H)$, so $g = f \cdot \mu$. Thus $f \cdot \mu \ast e_\alpha \rightarrow f \cdot \mu$ in the $\sigma(L^\infty (H) , L^\infty (H)^*)$-topology. Hence, $f \cdot \mu \in \mathcal{O}_L(f)$ and consequently $\mathbb{O}_L(f)\subseteq\mathcal{O}_L(f)$ which implies that $\mathbb{O}_L(f)$ is relatively weakly compact in $C_b(H)$. Therefore, by Lemma \ref{lem15}, $f \in wap(H).$ Thus, $WAP(L^1(H))\subseteq wap(H).$ Now the other way inclusion is a consequence of Lemma \ref{lem15}. 
	\end{proof}
	As an immediate consequence of \cite[Proposition 4.2.8]{ska89}   we obtain the following lemma.
	\begin{lemma}
		Let $H$ be a locally compact hypergroup with a Haar measure. Then $B(H) \subset wap(H)$.
	\end{lemma}

Here is the main result of this section. It has been known since the work of Arens \cite{are51} that there are two Banach algebra products on the bidual $L^1(H)^{**}$, each extending the convolution  on $L^1(H)$, and $L^1(H)$ is called {\it Arens regular} if those products coincide.  According to \cite[Theorem 3.14]{dal05}, the Banach algebra $L^1(H) $ is Arens regular if and only if $WAP(L^1(H))=L^\infty (H) $.
	
	\begin{theorem}\label{CPBPFS}
		Let $ H $ be an ultraspherical hypergroup. Then the following statements are equivalent:
            
                \item[(i)] Every subset of $ H $ belongs to $ \sR(H) $.
			    \item[(ii)] $L^1(H)$ is Arens regular.
			    \item[(iii)] $H$ is finite.
			    \item[(iv)] The Fourier-Stieltjes algebra $ B(H) $ has power boundedness property.
            
	\end{theorem}
	\begin{proof}
		(i) $\Longrightarrow$ (ii). First note that, by hypothesis and  Lemma \ref{lem7}, $\1_{\{\dot e\}}\in B(H).$ This implies that $H$ is discrete and hence $\sR(H)=\sR_c(H)$. 
		Now, let $ \dot E $ be any subset of $ H $. Applying hypothesis and  Lemma \ref{lem7} again, we get $ \1_{\dot E} \in B(H) \subseteq WAP(L^1(H)) $. Now density of simple functions in $ L^\infty (H)$ implies that $WAP(L^1(H)) = L^\infty (H) ,$ which in turn implies that $ L^1(H) $ is Arens regular.
		
		(ii) $\Longrightarrow$ (iii). This follows from  \cite[Theorem 5.2.3]{ska89}.
  
            (iii) $\Longrightarrow$ (i) is obvious.
		
		(iii) $\Longrightarrow$ (iv). If $ H $ is finite, then clearly $B(H) = A(H)$ and hence, by Theorem \ref{CPBP}, $B(H)$ has the power boundedness property.
		
		(iv) $\Longrightarrow$ (iii). 
		Suppose that $B(H)$ is power bounded. Then $A(H) $ is also power bounded. Thus, by Theorem \ref{CPBP}, $H$ is discrete. Towards a contradiction, assume that $H$ has an infinite countable subhypergroup $K.$ As $B(H)$ is power bounded, by Lemma \ref{lem8}, $B(K)$ also possesses the power boundedness property. Note that, by Lemma \ref{lem6}, for any subset $ \dot E $ of $ K $ there exists $ u \in B(K) $ with $ \Vert u \Vert_\infty =1 $ such that $ F_u = \dot E $. Now since $ u $ is power bounded, by Corollary \ref{cor3},  $ \dot E \in \sR_c(K) $. Thus $ K $ is finite, a contradiction. Hence, $ H $
		is finite.
	\end{proof}
	
	\subsection{Power boundedness of $I(\dot{E})$}
	
	Our next result deals with the power boundedness of the ideal $k(\dot B).$ This is an analogue of \cite[Proposition 2.5]{klu11}.
	
	\begin{proposition} \label{por4}
		Let  $ H $ be an ultraspherical hypergroup  and $ \dot E $ a closed subset of  $ H $. If the ideal $k(\dot E)$
		has the power boundedness property, then every compact subset of $H \setminus \dot E$ belongs to  $ \sR_c(H) $.
	\end{proposition}
	\begin{proof}
		Let $ \dot C $ be a compact subset of $ H\setminus \dot E $. Then, by \cite[Lemma 10.1 C]{jew75}, there exists an open $ \sigma $-compact subhpergroup $ K $ of $ H $ which contains $ \dot C $. Since $ K \setminus \dot C $ is $ \sigma $-compact, by Lemma \ref{lem6}, there exists $ u \in B(K) $ with $ \Vert u \Vert_\infty =1 $ such that $ \dot F_u =\dot C $. Now, use Lemma \ref{lem8},  to get  $ u_0 \in B(H) $ that is $ u $ on $ K $ and vanishes outside $ K $. Also, using regularity of $ A(H) $, one can find $ v\in A(H) $ such that $ 0\leq v\leq 1,$ $v \equiv 1$ on $\dot C$ and $\supp(v) \cap \dot E = \emptyset.$ Let $u_1 = v u_0  $. Then $ u_1 \in k(\dot E), \Vert u_1 \Vert _\infty =1 $ and $ \dot F_{u_1}= \dot C$. By hypothesis,   $ u_1 $ is power bounded and hence, by Lemma \ref{thm3},  $\dot C\in \sR_c (H).$
	\end{proof}
	
	We shall end this section with a result which is of independent interest. This result is about the existence of bounded approximate identity in the ideals $I(\dot E).$
	\begin{lemma} \label{lem11}
		Let $H$ be an ultraspherical hypergroup associated to an amenble  locally compact group $G$. If $ \dot E \in \sR_c(H) $, then
			\item[(i)] $\dot E  $ is a set of synthesis for $A(H).$
			\item[(ii)] $I(\dot E) $ has a bounded approximate identity.
	\end{lemma}
	\begin{proof}
		(i). This follows from \cite [Theorem 3.1]{dkl14}.\\
		(ii). Since $ p^{-1}(\dot E) \in \sR_c(G)$,  by \cite[Lemma 2.2]{fkl03} the ideal $ I (p^{-1}(\dot E)) $ has a bounded approximate identity.  Let $ v \in I (p^{-1}(\dot E)) $ and let $ v_0 $ be the element of $ A(H) $ associated with $ \pi(v) $. Then for each $ \dot x \in E $, we have
		$$v_0 (\dot x) = \pi (v) (x) = \int_{\mathcal O_x}v(z) d\pi^*(\delta_x)(z),$$
		since $\mathcal O_x \subseteq  p^{-1}(\dot E) $.
		Thus, we can identify $ \pi (I (p^{-1}(\dot E)))  $ with $I(\dot E) $. Suppose that $ (e_\alpha)_\alpha $ is a bounded approximate identity of $I (p^{-1}(\dot E))$ and let  $ (\dot e_\alpha)_\alpha $ be its  associated net in $ A(H) ,$ i.e., $\tilde {\dot e}_\alpha= \pi(e_\alpha) .$ Then, for each $ w \in I(\dot E) $, we have,
		\begin{eqnarray*}
		\Vert  w \dot e_\alpha - w\Vert_{A(H)} &=& \Vert \tilde w \tilde {\dot e}_\alpha - \tilde w\Vert_{A(G)} \\ &=&
		\Vert  \pi( \tilde w  e_\alpha - \tilde w)\Vert_{A(G)} \\ &\leq&\Vert   \tilde w  e_\alpha - \tilde w\Vert_{A(G)}\rightarrow 0.
		\end{eqnarray*}  
		This completes the proof.
	\end{proof}
	
	\section*{Acknowledgement}
        The second author would like to thank the Science and Engineering Board, India, for the MATRICS project fund with the Project No:MTR/2018/000849.
	
	\bibliographystyle{amsalpha}
	\bibliography{ref}

\end{document}